\documentclass{amsart}

 \usepackage{tikz}
 \usepackage{natbib}
\usepackage{amsmath,amsthm,amsfonts,amssymb,fancybox,float}
 \usepackage{framed}
 \usepackage{color,soul}

\newcommand{\commenthide}[1]{}

\def\pp#1{\left(#1\right)}
\def\spp#1{(#1)}
\def\bb#1{\left[#1\right]}
\def\sbb#1{[#1]}

\def\wt#1{\widetilde{#1}}
\def\summ#1#2#3{\sum_{#1 = #2}^{#3}}
\def\prodd#1#2#3{\prod_{#1 = #2}^{#3}}
\newcommand{\ind}{{\bf 1}}
\def\indd#1{{\ind}_{\{#1\}}}

\def\R{{\mathbb R}}

\def\topp#1{^{(#1)}}

\def\Bex{\mathbb B^{ex}}
\def\B{\mathbb B}
\def\Bme{\mathbb B^{me}}
\newcommand{\Bnu}{\mathbb{B}\topp\nu}
\newcommand{\hidecomment}[1]{}
\newtheorem{theorem}{Theorem}[section]
\newtheorem{corollary}[theorem]{Corollary}

\theoremstyle{definition}

\theoremstyle{remark}
\newtheorem{remark}[theorem]{Remark}
\theoremstyle{remark}

\usepackage{graphicx,fancybox,pgf,amsmath,float,amssymb}
\usepackage{fancybox,graphicx,dsfont,pgf}
\usepackage{subfigure,color,wrapfig}

\usepackage{latexsym}

\usepackage{dsfont}
\newcommand{\RR}{\mathds{R}}

\newcommand{\EE}[1]{\mathds{E}\left[#1\right]}
\newcommand{\EEx}[2]{\mathds{E}_{#1}\left[#2\right]}

\def\PP{\mathds{P}}

\def\<{\langle}
\def\>{\rangle}

\numberwithin{equation}{section}

\title[Laplace transform of Brownian excursion and meander]{Dual representations of
 Laplace transforms of
 Brownian excursion and
generalized
   meanders
}
\author{W\l odzimierz Bryc}
\address
{
W\l odzimierz Bryc\\
Department of Mathematical Sciences\\
University of Cincinnati\\
2815 Commons Way\\
Cincinnati, OH, 45221-0025, USA.
}
\email{wlodzimierz.bryc@uc.edu}

\author{Yizao Wang}
\address
{
Yizao Wang\\
Department of Mathematical Sciences\\
University of Cincinnati\\
2815 Commons Way\\
Cincinnati, OH, 45221-0025, USA.
}
\email{yizao.wang@uc.edu}

\begin{document}\sloppy
\begin{abstract}

The Laplace transform of %
the $d$-dimensional distribution of Brownian excursion is expressed as
the Laplace transform of
the $(d+1)$-dimensional
 distribution of an auxiliary Markov process, started from a $\sigma$-finite
measure
and with the roles of arguments and times interchanged.
A similar identity holds for the Laplace transform of a generalized Brownian  meander, which is expressed as  the Laplace transform of
the same auxiliary Markov process, with a different initial law.

\end{abstract}
\maketitle

\section{Introduction}\label{Sec:Into}

Consider the Brownian excursion standardized to have length 1 and conditioned to be positive.
 This is also the Brownian bridge conditioned to stay positive, or the 3-dimensional Bessel process conditioned to hit zero at
 $t=1$, and its %
 finite-dimensional distributions are given by formula \eqref{eq:RevuzYor} below.
We let  $\Bex =(\Bex_t)_{t\in[0,1]}$ denote this process throughout.
Brownian excursion has been extensively investigated in the literature. See for example \citet{bertoin94path} and
\citet{revuz99continuous}. It also appears in asymptotic analysis
   of various combinatorial problems, see for example
   \citet{pitman06combinatorial} and \citet{janson2007brownian}.

The purpose of this note is to introduce a   ``dual representation" that ties the
Laplace transforms of finite-dimensional distributions of
 Brownian excursion and another Markov process, denoted by $(X_t)_{t\geq 0}$ throughout,
with
state space $[0,\infty)$ and  transition probabilities
\begin{equation*}
   \PP(X_t \in dy\mid X_s =x)=p_{t-s}(x,y) dy  ,\; 0\leq s<t, x\geq 0
\end{equation*}
with
\begin{equation}
  \label{X-transitions}
 p_t(x,y) =\frac{2  t \sqrt{y}}{\pi
  \bb{ (y-x)^2+2    (x+y)t^2+t^4
 }},\;t>0, y\geq 0.
\end{equation}
This is a positive self-similar Markov process
 that arises as the tangent process at the boundary of support of
 so-called $q$-Brownian motions and $q$-Ornstein--Uhlenbeck processes; see \citet[Proposition 2.2]{Bryc-Wang-2015} and \citet[Theorem 3.1]{wang2016extremes}.
It can also be obtained from the construction in \citet{biane98processes} applied to the
1/2-stable free L\'evy process
by including appropriate drift.  The
derivation of the transition probability density function, following Biane's approach,
can be found in \citet[Section 3]{Bryc-Wang-2015}.

Our main result is the following
identity which was needed in \cite{Bryc-Wang-2017ASEP},
where we investigated, by essentially computing the Laplace transforms, the fluctuations of asymmetric simple exclusion processes with open boundaries
 in the steady state. These processes are representative non-equilibrium models
 \cite{derrida2007non,derrida2004asymmetric}
 that have attracted much
 attentions recently in probability and mathematical physics.

Let $\EEx{x}{\cdot}$ denote the expectation with respect to the law of $(X_t)_{t\ge0}$ starting at
$X_0 =
 x>0$.
 \begin{theorem}
  \label{T1.5}
For $d\in\mathbb N$, let  $s_0=0<s_1<s_2<\dots<s_d$ and
$t_0=0\le t_1<\dots<t_d\le 1=t_{d+1}$. Then,
\begin{multline}
  \label{Excursion-dual}
  \EE{\exp\pp{-\summ k1d (s_k-s_{k-1})\Bex_{t_k}}} \\=
 \frac1{\sqrt{2\pi}} \int_0^\infty \EEx{x}{\exp \pp{-\frac12\sum_{k=0}^{d}
(t_{k+1}-t_{k})X_{s_k}}}\sqrt{x}\,dx.
\end{multline}
\end{theorem}

The left-hand side of~\eqref{Excursion-dual} is the Laplace transform of the joint distribution of Brownian excursion, while
the right-hand side is the Laplace transform of the
process $(X_t)_{t\geq 0}$
 with the arguments and time indices
interchanged. %
 On the right-hand side, the initial distribution of
 $(X_t)_{t\geq 0}$
  is  the stationary
$\sigma$-finite measure $(x/(2\pi))^{1/2}\indd{x>0}dx$.

We are aware of only a couple of results that connect Laplace transforms of stochastic processes by interchanging the argument and time parameters. One such result  is
the formula for the joint generating function of the finite Asymmetric Simple Exclusion Process in \cite[Theorem 1]{Bryc-Wesolowski-2015-asep}.
Another result of this type is  the formula \citet[Eq.~(2)]{Bertoin-Yor2001subordinators} for the univariate Mellin transform of a positive self-similar Markov process,
see also  \citet{hirsch2012temporally}.

As an immediate consequence of Theorem~\ref{T1.5} however, a family of such dualities can be derived easily for {\em
generalized Brownian meanders} defined in \eqref{eq:Gen-meander} below, including some  of the generalized Bessel meanders considered
in \citep[Section 3.7]{Mansuy-Yor-2008}.
The dualities, see   Theorem~\ref{T-gen-meander} and Corollary \ref{coro:3}, have
similar forms as in \eqref{Excursion-dual}, but they differ in the choice of the $\sigma$-finite measure for the initial law on
the right-hand side of~\eqref{Excursion-dual}.

Due to  its importance,
here we state the duality formula  for  Brownian meander $(\Bme_t)_{0\leq t\leq 1}$, which is a special case $\delta = 1$ of %
Corollary \ref{coro:3} below.  Recall that  Brownian meander can be defined as the
Brownian motion conditioned to stay positive over the time interval $[0,1]$.  See for example
\citet{bertoin94path,Pitman-1999,pitman06combinatorial}.
We also need this formula for our analysis in \cite{Bryc-Wang-2017ASEP} for asymmetric simple exclusion processes.

\begin{theorem}
  \label{coro:excursion}
For
 $d\in\mathbb N$, let  $s_0=0<s_1<s_2<\dots<s_d$ and $t_0=0\leq t_1<t_2<\dots
 <t_d\le 1=t_{d+1}$. Then
\begin{multline}
  \label{Meander-dual}
  \EE{\exp\left(-\sum_{k=1}^d (s_{k}-s_{k-1})\Bme_{1-t_k}\right)}\\=
 \frac1{\sqrt{2\pi}} \int_0^\infty \EEx{x}{\exp \pp{-\frac12\sum_{k=0}^{d}
(t_{k+1}-t_{k})X_{s_k}}}\frac{1}{\sqrt{x}}\,dx .
\end{multline}
\end{theorem}

The paper is organized as follows. In Section~\ref{Sec:Exc} we recall some known facts about Brownian excursion and prove
Theorem~\ref{T1.5}. In Section~\ref{Sec:randomized} we prove the dual representations for
generalized Brownian meanders, with
   Brownian meander as a special case.

\section{Proof of dual representation of Brownian excursion}
\label{Sec:Exc}
We first recall some facts on Brownian excursion $\Bex$.
 For our purposes it is convenient to define it as a
 Markov process  that starts at $\Bex_0=0$, ends at $\Bex_1=0$,
  and has transition probabilities
 \begin{equation*}\label{Transitions-Excursion}
P(\Bex_t\in dy\mid \Bex_s=x)=\begin{cases}
  \sqrt{8\pi} \ell_t(y)\ell_{1-t}(y)  &\mbox{ if $s=0<t<1, x=0, y> 0$}\\
   \\
\displaystyle g_{t-s}(x,y)\frac{\ell_{1-t}(y)}{\ell_{1-s}(x)}
       &\mbox{ if $0<s<t<1, x>0,y> 0$},\\
   \end{cases}
 \end{equation*}
with
\begin{equation*}
  \label{RY-ell}
  \ell_t(y) = \frac1{\sqrt{2\pi t^3}}\cdot y\exp\pp{-\frac{y^2}{2t}}\indd {y>0}
\end{equation*}
and
\begin{equation*}
  \label{RY-q}
  g_t(y_1,y_2) = \frac1{\sqrt{2\pi t}}\bb{\exp\pp{-\frac{(y_1-y_2)^2}{2t}} - \exp\pp{-\frac{(y_1+y_2)^2}{2t}}}\indd{y_1,y_2>0}.
\end{equation*}
Equivalently, the joint probability density function at time points $0<t_1<\cdots<t_d<1$  is
\begin{equation}
  \label{eq:RevuzYor}
    f_{t_1,\dots,t_d}(y_1,\dots,y_d) = \sqrt{8\pi}\ell_{t_1}(y_1) \ell_{1-t_d}(y_d)\prodd k1{d-1}g_{t_{k+1}-t_k}(y_k,y_{k+1}).
\end{equation}
See %
\citet[page 76]{ito125jr}, \citet{durrett1977weak}, %
 or
\citet[page 464]{revuz99continuous}.

\begin{proof}[Proof of Theorem~\ref{T1.5}]
The proof consists of rewriting the right-hand side of \eqref{Excursion-dual}.
We first assume $t_1>0$ and $t_d<1$.
With $t_0=0$, we have, using transition density functions of the process $X$ and letting $d\vec x$ denote $dx_0\cdots dx_d$,
\begin{align}
  \int_0^\infty &\EEx{x_0}{\exp \pp{-\frac12\sum_{k=0}^d
(t_{k+1}-t_{k})X_{s_k}}}\sqrt{x_0}\,dx_0  \nonumber\\
&= \int_{\R_+^{d+1}}d\vec x  \sqrt{x_0} \exp\pp{-\frac12\summ k0d (t_{k+1}-t_{k})x_k}\prodd k1d p_{s_k-s_{k-1}}(x_{k-1},x_{k})\nonumber
\\
&=  2^{d+1}\int_{\R_+^{d+1}}d\vec x  \exp\pp{-\frac12\summ k0d (t_{k+1}-t_{k})x_k^2} \times x_0^2  \prod_{k=1}^d  x_kp_{s_k-s_{k-1}}(x_{k-1}^2,x_{k}^2),\label{iiint}
\end{align}
where we apply in the last equality changes of variables from $x_k$ to $x^2_k$, $k=0,\dots,d$.
We derive a different expression of the products of transition density functions, thanks to the factorization. Namely,  recalling~\eqref{X-transitions}, we write
\begin{align*}
 p_s(x^2,y^2) &=\frac{2s y}{\pi (s^2+(y-x)^2)(s^2+(y+x)^2)}\\
&=\frac{1}{2\pi x}\left(\frac{s}{s^2+(y-x)^2}-\frac{s}{s^2+(y+x)^2} \right).\end{align*}
We now use the elementary Laplace transform
$$
\frac{s}{s^2+a^2}=\int_0^\infty e^{-s y} \cos (a y) dy$$
to get
\begin{align*}
 y p_s(x^2,y^2) &=\frac{y}{2\pi  x}\int_0^\infty e^{-s z} \bb{\cos ((y-x) z)-\cos ((y+x) z)}dz\\
 &= \frac{1}{\pi} \frac{y}{  x}\int_0^\infty e^{-s z} \sin (xz)\sin(yz)dz.
\end{align*}
Therefore, writing $d\vec y = dy_1\cdots dy_d$,
\begin{multline*}
x_0^2 \prod_{k=1}^d  x_kp_{s_k-s_{k-1}}(x_{k-1}^2,x_{k}^2) = \frac1{\pi^d}\int_{\R_+^d}d\vec y \, \exp\pp{-\summ k1d(s_k-s_{k-1})y_k} \\
\times x_0\sin(x_0y_1) \times x_d \sin(x_dy_d) \times \prodd k1{d-1}\sin(x_ky_{k+1})\sin(x_ky_k).
\end{multline*}
Then, noticing that the integrand is non-negative and interchanging the order of integrations, we rewrite the right-hand side
of \eqref{iiint} as %
\begin{align}
  \frac{2^{d+1}}{\pi^d} & \int_{\R_+^{d}}d\vec y \,\exp\pp{-\summ k1d(s_k-s_{k-1})y_k}\label{eq:interchanging}\\
& \times \int_0^\infty x_0e^{-t_1x_0^2/2}\sin(x_0y_1)dx_0
\times \int_0^\infty x_de^{-(1-t_d)x_d^2/2}\sin(x_dy_d)dx_d\nonumber\\
&\times \prodd k1{d-1}\int_0^\infty e^{-(t_{k+1}-t_k)x_k^2/2}\sin(x_ky_{k+1})\sin(x_ky_k)dx_k.\nonumber
\end{align}
By straightforward calculation we have
\begin{equation*}\label{eq:sin}
\int_0^\infty x e^{-t x^2/2}\sin (xy) dx  = \frac yt\int_0^\infty e^{-tx^2/2}\cos(xy)dx = \pi\ell_t(y).
\end{equation*}
In the last step above we used the %
formula for the cosine transform
\begin{equation*}\label{eq:cos}
\int_0^\infty e^{-tx^2/2}\cos(ax)dx
= \sqrt{\frac\pi2}\frac1{t^{1/2}}e^{-a^2/(2t)},
\end{equation*}
which follows from the characteristic function of a Gaussian distribution with mean zero and variance $1/t$. We also have
\begin{align*}
\int_0^\infty & e^{-t x^2/2}\sin(xy_1)\sin(xy_2)dx \\
\ & = \frac12\int_0^\infty e^{-tx^2/2} \bb{\cos(x(y_1-y_2))-\cos(x(y_1+y_2))}dx\\
&= \sqrt{\frac\pi8}\frac1{t^{1/2}}\pp{e^{-(y_1-y_2)^2/(2t)} - e^{-(y_1+y_2)^2/(2t)}}  = \frac\pi2 g_t(y_1,y_2).
\end{align*}
Thus, recalling~\eqref{eq:RevuzYor}, expression~\eqref{eq:interchanging} now becomes,
\begin{multline*}
 4\pi\int_{\R^d}d\vec y\, \exp\pp{-\summ k1d(s_k-s_{k-1})y_k}
 \ell_{t_1}(y_1)\ell_{1-t_d}(y_d)\prodd k1{d-1}g_{t_{k+1}-t_k}(y_k,y_{k+1}) \\
= \sqrt{2\pi}\EE{\exp\pp{-\summ k1d (s_k-s_{k-1}) \Bex_{t_k}}}.
\end{multline*}
We have proved the theorem with $0<t_1,t_d<1$.

Now assume $t_d = 1$ and $t_1>0$. %
Recall that $\Bex_1 = 0$. If $d=1$, then the left-hand side of~\eqref{Excursion-dual} becomes 1, and the right-hand side becomes
$(\sqrt{2\pi})^{-1}\int_0^\infty e^{-x/2}\sqrt xdx = 1$.
If $d>1$, %
we see that  the desired identity is reduced to the same type of identity with
fewer arguments, namely with $0<s_1<\cdots<s_{d-1}$ and $0<t_1<\cdots<t_{d-1}<1 = t_d$. Such an identity has been proved in the
first part of the proof.

It remains to prove
the theorem
for $t_1 = 0$. If $d=1$, then the left-hand side
 of~\eqref{Excursion-dual}  is 1 and the right-hand side equals
\[
\frac1{\sqrt{2\pi}}\int_0^\infty \EEx{x}{e^{-X_{s_1}/2}}\sqrt xdx
=
 \frac1{\sqrt{2\pi}}\int_0^\infty e^{-x/2}\sqrt xdx = 1,
\]
where we used the fact that $X$ is a Markov process with stationary distribution $\sqrt x \indd{x\ge 0}dx$ (as $\sqrt xp_s(x,y) = \sqrt y p_s(y,x)$ for $x,y\ge 0$, $s>0$).
For  $d\ge 2$, consider
$\wt s_k := s_{k+1}-s_1$, $\wt t_k := t_{k+1}, k=0,\dots,d-1$ and $\wt t_d := 1$. In this way, the left-hand side of~\eqref{Excursion-dual} becomes
$\mathds E\sbb{\exp\spp{-\summ k1{d-1} (\wt s_k-\wt s_{k-1})\Bex_{\wt t_k}}}$, and
the right-hand side becomes
\begin{multline*}
 \frac1{\sqrt{2\pi}} \int_0^\infty \EEx{x}{\exp \pp{-\frac12\sum_{k=0}^{d-1}
(\wt t_{k+1}-\wt t_{k})X_{\wt s_k+s_1}}}\sqrt{x}\,dx\\
 =  \frac1{\sqrt{2\pi}} \int_0^\infty \EEx{x}{\exp \pp{-\frac12\sum_{k=0}^{d-1}
(\wt t_{k+1}-\wt t_{k})X_{\wt s_k}}}\sqrt{x}\,dx.
\end{multline*}
Above in the last step we used the stationarity of $X$ again.
Since necessarily $\wt t_1
=t_2
>0$, when $d\ge 2$ the desired identity %
\eqref{Excursion-dual}
becomes an identity for Laplace transform of ($d-1$)-dimensional distribution of Brownian excursion at time points $0<\wt t_1<\cdots<\wt t_{d-1}\le 1$ that we have
already
proved before.
This completes the proof.
\end{proof}

\section{Application to generalized Brownian meanders}
\label{Sec:randomized} In this section, we let $\nu(dv)$ be a probability measure on $[0,1]$ such that $\nu(\{0\})=0$, and
consider the
generalized Brownian meander $(\Bnu_t)_{0\leq t\leq 1}$ defined by
\begin{equation}
  \label{eq:Gen-meander}
  \Bnu_t =\frac1{\sqrt V}\Bex_{Vt}, \; t\in[0,1],
\end{equation}
where $V$ is a random variable with law $\nu$  and  independent from $\Bex$.
Explicit examples,
including Brownian meander,
will be discussed
later  in  this section.

Noting that the function $v\mapsto e^{-x/v} v^{-3/2}$ is bounded on $(0,1]$  for $x>0$, we define
\begin{equation}
   \label{eq:phi_nu}
   \varphi_\nu(x)=\frac{\sqrt{x}e^{x/2}}{\sqrt{2\pi}}\int_{(0,1]} e^{-\tfrac{x}{2v}}v^{-3/2} \nu(dv).
\end{equation}

We have the following  dual representation of $\Bnu$. %
\begin{theorem}
  \label{T-gen-meander}
For %
$d\in\mathbb N$, let  $s_0=0<s_1<s_2<\dots<s_d$ and $t_0=0\leq t_1<t_2<\dots<t_d\le 1=t_{d+1}$.
Then
\begin{multline*}
  \EE{\exp\left(-\sum_{k=1}^d (s_{k}-s_{k-1})\Bnu_{1-t_k}\right)}\\=
 \int_0^\infty \EEx{x}{\exp \pp{-\frac12\sum_{k=0}^{d}
(t_{k+1}-t_{k})X_{s_k}}}\varphi_\nu(x)dx . %
\end{multline*}
\end{theorem}
\begin{proof}
The key step is the following identity
\begin{multline}
  \label{eq:pre-V}
  \EE{ \exp  \pp{-\summ k1d \frac{s_k-s_{k-1}}{\sqrt v}\Bex_{v(1-t_k)}} }
\\=
  \frac{1}{\sqrt{2\pi}}\int_0^\infty %
  \sqrt xe^{x/2}\frac{e^{-\tfrac x{2v}}}{v^{3/2} }
  \EEx{x}{\exp \pp{-\frac12\summ k0d (t_{k+1}-t_k)X_{ s_k} }} dx
\end{multline}
for $0<v\leq 1$.
Once we establish~\eqref{eq:pre-V}, integrating with respect to $\nu$ on both sides yields the desired result, by Fubini's theorem.

Now we prove~\eqref{eq:pre-V}.
We first
rewrite
\eqref{Excursion-dual}, with $t_k$ replaced by $1-t_k$, as %
processes  $(\Bex_t)_{t\in[0,1]}$ and $(\Bex_{1-t})_{t\in[0,1]}$ have the same law. So~\eqref{Excursion-dual} becomes
\begin{multline}\label{aadvark}
  \EE{\exp  \pp{-\summ k1d (s_k-s_{k-1})\Bex_{1-t_k}}}\\
 = \frac1{\sqrt{2\pi}} \int_0^\infty \EEx{x}{\exp \pp{-\frac12\summ k1d (t_{k+1}-t_k)X_{s_k} - \frac12(1-(1-t_1))X_0}}\sqrt x\, dx.
\end{multline}
This identity remains valid, if we replace the increasing sequence $(s_k)_{k=0,\dots,d}$ with %
 $(s_k/\sqrt{v})_{k=0,\dots,d}$ and replace the increasing sequence $(t_k)_{k=1,\dots,d+1}$ with %
  $(\wt t_k)_{k=1,\dots,d+1}$ with $\wt t_k := 1-v(1-t_k)$ (formally we just replace $s_k$ by $s_k/\sqrt v$ and $1-t_k$ by $v(1-t_k)$, but we also need to verify the same monotonicity of the replaced sequences, and that $\wt t_1 \ge 0, \wt t_d\le 1$). In this way we get
\begin{multline*}
  \EE{\exp  \pp{-\summ k1d \frac{s_k-s_{k-1}}{\sqrt v}\Bex_{v(1-t_k)}}}\\
 = \frac1{\sqrt{2\pi}} \int_0^\infty \EEx{x}{\exp \pp{-\frac12\summ k1d v(t_{k+1}-t_k)X_{\frac{s_k}{\sqrt v}} - \frac12(1-v(1-t_1))X_0}}\sqrt x \, dx.
 \end{multline*}
 Thanks to the self-similarity of process $X$,
 \[
 ((X_{\lambda s})_{s\ge 0}, \PP_x) \stackrel{d}= \pp{\lambda^2(X_s)_{s\ge 0},
 \PP_{x/\lambda^2}}, x>0, \lambda>0,
 \]
 where $\PP_x(\cdot) = \PP(\cdot\mid X_0 = x)$ (see \citet{Bryc-Wang-2015}), %
 we have
 \begin{align*}
 &\int_0^\infty \EEx{x}{\exp \pp{-\frac12\summ k1d v(t_{k+1}-t_k)X_{\frac{s_k}{\sqrt v}} - \frac12(1-v(1-t_1))X_0}}\sqrt xdx\\
 & = \int_0^\infty \EEx{vx}{\exp\pp{-\frac12\summ k1d(t_{k+1}-t_k)X_{s_k}-\frac12\pp{{1-v(1-t_1)}}\frac{X_0}v}}\sqrt xdx\\
 & = \int_0^\infty \EEx{vx}{\exp\pp{-\frac12\summ k0d(t_{k+1}-t_k)X_{s_k}-\frac12\pp{\frac1v -1}X_0}}\sqrt xdx\\
 & = \int_0^\infty \EEx{x}{\exp\pp{-\frac12\summ k0d(t_{k+1}-t_k)X_{s_k}}}e^{-x\pp{1/v -1}/2}\frac{\sqrt x}{v^{3/2}}dx.
 \end{align*}
This yields~\eqref{eq:pre-V}, and completes the proof.
\end{proof}

Next, we specialize  Theorem~\ref{T-gen-meander}  to   the case of  Beta  distribution.
 The formula shall involve the confluent hypergeometric function
\[
\psi(\alpha,\beta,x) = \frac1{\Gamma(\alpha)}\int_0^\infty e^{-xu}u^{\alpha-1}(1+u)^{\beta-\alpha-1}du, \alpha>0,
\beta\in\RR,
x>0.
\]
See for example~\citep[page 268]{lebedev65special}.

\begin{corollary}
\label{coro:gen-meander} For %
$d\in\mathbb N$, let  $s_0=0<s_1<s_2<\dots<s_d$ and $t_0=0\le t_1<\dots<t_d\le 1=t_{d+1}$. If
$\nu(dv)=v^{\alpha-1}(1-v)^{\beta-1}/B(\alpha,\beta)$ with $\alpha,\beta>0$, then
\begin{multline}
  \label{Meander-alpha-dual}
  \EE{\exp\left(-\sum_{k=1}^d (s_{k}-s_{k-1})\B\topp{\nu}_{1-t_k}\right)}
\\
  = \frac{\Gamma(\alpha+\beta)}{\sqrt {2\pi}\Gamma(\alpha)}   \int_0^\infty \EEx{x}{\exp \pp{-\frac12\sum_{k=0}^{d}(t_{k+1}-t_{k})X_{s_k}}}  \psi\pp{\beta,\frac{5}2-\alpha,\frac x2}\sqrt x \,dx.
\end{multline}
\end{corollary}

\begin{proof}
We apply Theorem~\ref{T-gen-meander} with
\begin{align}
\varphi_\nu(x) & =
\frac1{\sqrt{2\pi}}\frac{\sqrt x}{B(\alpha,\beta)}\int_0^1e^{-x(1/v-1)/2}v^{\alpha-5/2}(1-v)^{\beta-1}dv\nonumber\\
\label{eq:d+d'=3}
& = \frac1{\sqrt{2\pi}}\frac{\sqrt{x}}{B(\alpha,\beta)}\int_0^\infty e^{-ux/2}(1+u)^{3/2-(\alpha+\beta)}u^{\beta-1}du\\
& = \frac{\sqrt x\Gamma(\beta)}{\sqrt{2\pi} B(\alpha,\beta)}\psi\pp{\beta,\frac{5}2-\alpha,\frac x2},\nonumber
\end{align}
where in the second step we applied a change of variable $u = 1/v-1$.
\end{proof}

Our results are also related to
 generalized Bessel meanders that were introduced in  \citet[Sections 3.6 and 3.7]{Mansuy-Yor-2008}.
There are three equivalent ways to define a generalized Bessel  meander %
$\B\topp{\delta,\delta'}$  for $\delta,\delta'>0$.
(This is the process under the law $M^{\delta,\delta'}$ in \citet{Mansuy-Yor-2008}, and strictly speaking the authors did not give it
 a name there but only mentioned `generalized meanders' in the section titles.)
Firstly, it can be defined as a
randomized Bessel %
bridge
of
dimension $\delta+\delta'$,
using Beta$(\delta/2,\delta'/2)$ distribution.
For this approach, see \citet[Theorem 3.12]{Mansuy-Yor-2008}. Secondly, it can be defined as $((R_t^2+
(R_t')^2)^{1/2})_{t\in[0,1]}$, where $R$ is a $\delta$-dimensional Bessel bridge (pinned down at $R_0 = 0$ and $R_1 = 0$),
$R'$ is a $\delta'$-dimensional Bessel process starting from 0, and $R$ and $R'$ are independent.
Thirdly, the law $M^{\delta,\delta'}$ of $\B\topp{\delta,\delta'}$, viewed as a probability measure on $C([0,1])$,  is absolutely continuous with respect to the %
probability measure
 $P_0^{{\rm BES}(\delta+\delta')}$ on $C([0,1])$ induced by a $(\delta+\delta')$-dimensional Bessel process starting from 0; more precisely
\begin{equation}\label{eq:Mdd}
M^{\delta,\delta'} = \frac{c_{\delta,\delta'}}{X_1^\delta}P_0^{{\rm BES}(\delta+\delta')},\quad \delta,\delta'>0,
\end{equation}
with %
\begin{equation}\label{eq:cdd}
c_{\delta,\delta'} = M^{\delta,\delta'}(X_1^\delta) = \frac{2^{\delta /2}\Gamma((\delta+\delta')/2)}{\Gamma(\delta'/2)}.
\end{equation}
 Here, for $\omega$ from the canonical space $C([0,1])$, $X_1(\omega) = \omega_1$.
For the second and third characterizations, see \citet[Theorem 3.9]{Mansuy-Yor-2008};
the corresponding formula \eqref{eq:cdd} in Theorem 3.9 therein has a typo and is corrected here.

The generalized Bessel meanders $\B^{\delta,\delta'}$ and the generalized Brownian meanders $\B\topp\nu$ with Beta distribution $\nu$ in Corollary \ref{coro:gen-meander} are in general different processes. However,
since a Bessel bridge of dimension 3 is a Brownian excursion, generalized Bessel meanders %
\eqref{eq:Mdd} with \begin{equation}\label{eq:dd}
\delta\in(0,3) \mbox{ and } \delta' = 3-\delta,
\end{equation}
  become a special
case of the generalized Brownian meanders introduced in \eqref{eq:Gen-meander}. This case covers a couple of  examples
investigated in the literature.
In particular,
it is known that $\B\topp {1,2}$ is the Brownian meander $\Bme$ and in this case the relation~\eqref{eq:Mdd}
is
 originally due to \citet{imhof84density}. The process $\B\topp {2,1}$, known as {\em Brownian co-meander}, has also been investigated before,
 and the relation~\eqref{eq:Mdd}
  is due to \citet{biane87processus}. See also~\citet[Theorem 7.4.1]{Yen-Yor-2013}.
Our Corollary \ref{coro:gen-meander}
simplifies
and takes the following form.
\begin{corollary}\label{coro:3}
In the notations of %
Corollary \ref{coro:gen-meander}, we have for all $\delta\in(0,3)$,
\begin{multline}
  \label{Meander-alpha-dual1}
  \EE{\exp\left(-\sum_{k=1}^d (s_{k}-s_{k-1})\B\topp{\delta,3-\delta}_{1-t_k}\right)}
\\= \frac1{2^{\delta/2}\Gamma(\delta/2)} \int_0^\infty \EEx{x}{\exp \pp{-\frac12\sum_{k=0}^{d}(t_{k+1}-t_{k})X_{s_k}}}x^{\delta/2-1} dx.
\end{multline}
\end{corollary}

\begin{proof}
From \eqref{eq:d+d'=3}
with $\alpha=\delta/2$, $\beta=3/2-\alpha$   it follows that
\begin{align*}
\varphi_\nu(x) & = \frac{\sqrt x}{\sqrt{2\pi}}\frac{\Gamma(\alpha+\beta)}{\Gamma(\alpha)\Gamma(\beta)}\int_0^\infty e^{-ux/2}u^{\beta-1}
du
 = \frac{2^{\beta}\Gamma(\alpha+\beta)}{\sqrt{2\pi}\Gamma(\alpha)}x^{1/2-\beta}\\& = \frac{x^{\delta/2-1}}{2^{\delta/2}\Gamma(\delta/2)}.
\end{align*}
\end{proof}

\begin{remark}
When $d=1$,
expression
\eqref{Meander-alpha-dual1} becomes
\[
  \EE{ e^{-s
  \B\topp{\delta,3-\delta}_{t}}}=%
 \frac1{2^{\delta/2}\Gamma(\delta/2)} \int_0^\infty\EEx{x}{e^{-tX_s/2}} e^{-(1-t)x/2}x^{\delta/2-1}\,dx,
\]
This differs from \citet[Eq.~(2)]{Bertoin-Yor2001subordinators}, who   developed a dual representation of Laplace transforms
of univariate distributions in the form of univariate integrals
\[
\EE{ e^{-s U_t}}=\EE{ e^{-t R_s}}, s,t>0.
\]
 But, clearly, there is some similarity in how the    roles of arguments and times are interchanged.
\end{remark}
\begin{remark}\label{Joseph}
It is natural to interpret %
Beta$(0,3/2)$
 and Beta$(3/2,0)$  as degenerate laws.
The limiting case of %
 Beta$(\delta/2,3/2-\delta/2)$ as %
$\delta\uparrow3$ is  $V=1$,
 so with
  $\B\topp{3,0}=\Bex$
Corollary \ref{coro:3} can be viewed as an extension
 of Theorem~\ref{T1.5}.
On the other hand, as $\delta\downarrow 0$,
one can check that
$B\topp{ \delta,3-\delta}$
converges to the 3-dimensional Bessel process.
In this limit, the  right-hand side of \eqref{Meander-alpha-dual} becomes an undefined expression. %
The 3-dimensional Bessel process is also considered in \citet[Theorem 6.1]{hirsch2012temporally}.
Laplace transforms of squared Bessel processes have been investigated in \citet{pitman82decomposition}. However, we do not see immediate connection between our
identities
 to the results there.
\end{remark}

\subsection*{Acknowledgement}
The authors thank Jim Pitman for insightful comments  on %
several early versions of the paper, and a few key references including in particular \citet{Mansuy-Yor-2008},
which helped us improve significantly the paper. The authors also thank Joseph Najnudel for several inspiring discussions.
  WB's research was supported in part by the Charles Phelps Taft Research Center at the University of Cincinnati. YW's research was partially supported by NSA grant H98230-16-1-0322 and Army Research Laboratory grant W911NF-17-1-0006.

\bibliographystyle{apalike}
\bibliography{exclusion2017}%
\end{document}